\documentclass{article}
\usepackage{graphicx} % Required for inserting images

\usepackage[utf8]{inputenc}
\usepackage{enumitem,amssymb,url}
\usepackage{amsmath,amsxtra,amssymb,amsfonts,latexsym, mathrsfs, dsfont,bm,enumitem, amsthm, mathtools}
\usepackage{enumitem}
\usepackage{color,soul}
\usepackage{cite}
\usepackage{hyperref}
\usepackage{geometry}[margin=1in]

\usepackage{enumitem,amssymb,url}
\newlist{todolist}{itemize}{2}
\setlist[todolist]{label=$\square$}
\usepackage{pifont}
\newtheorem{theorem}{Theorem}[section]
\newtheorem{proposition}[theorem]{Proposition}
\newtheorem{conjecture}[theorem]{Conjecture}

\newtheorem{problem}[theorem]{Problem}

\newtheorem{corollary}[theorem]{Corollary}
\newtheorem{lemma}[theorem]{Lemma}

\theoremstyle{definition}
\newtheorem{definition}{Definition}[section]

\theoremstyle{definition}

\newcommand{\N}{\mathbb{N}}

\newcommand{\R}{\mathbb{R}}

\newcommand{\da}{\partial^\alpha}

\newcommand{\mb}{\bar{m}}

\title{On $C^m$ Solutions to Systems of Linear Inequalities}
\author{Garving K. Luli and Kevin O'Neill}
\begin{document}

\maketitle

\begin{abstract}
    Recent work of C. Fefferman and the first author \cite{LULI-3} has demonstrated that the linear system of equations
    \begin{equation*}
    \sum_{j=1}^M A_{ij}(x)F_j(x)=f_i(x)\hspace{.2in} (i=1,...,N),
\end{equation*}
has a $C^m$ solution $F=(F_1,...,F_M)$ if and only if $f_1,...,f_N$ satisfy a certain finite collection of partial differential equations. Here, the $A_{ij}$ are fixed semialgebraic functions.

In this paper, we consider the analogous problem for systems of linear \textit{inequalities}:
\begin{equation*}
    \sum_{j=1}^M A_{ij}(x)F_j(x)\le f_i(x)\hspace{.2in} (i=1,...,N).
\end{equation*}
Our main result is a negative one, demonstrated by counterexample: the existence of a $C^m$ solution $F$ may not, in general, be determined via an analogous finite set of partial differential inequalities in $f_1,...,f_N$.
\end{abstract}

\section{Introduction}

Fix $m,M,n,N\in\mathbb{N}$. Consider the system of linear equations given by
\begin{equation}\label{eq:system of linear equalities}
    \sum_{j=1}^M A_{ij}(x)F_j(x)=f_i(x)\hspace{.2in} (i=1,...,N),
\end{equation}
where the $A_{ij}$ and $f_i$ are given functions on $\mathbb{R}^n$, while $F_{1},\cdots ,F_{M}\in C^{m}\left( \mathbb{R}%
^{n}\right) $ are unknown functions to be solved for fixed $m$.\footnote{$C^{m}\left( \mathbb{R}^{n}\right) $
denotes the vector space of $m$-times continuously differentiable functions $%
\mathbb{R}^{n}$, with no growth conditions assumed at infinity.\ Similarly, $%
C^{m}\left( \mathbb{R}^{n},\mathbb{R}^{D}\right) $ denotes the space of all
such $\mathbb{R}^{D}$-valued functions on $\mathbb{R}^{n}$.} Notice that we do not impose any regularity conditions on $A_{ij}$ and $f_i$; in fact, they may be discontinuous functions, e.g., indicator functions on closed sets. While elementary linear algebra can be used to find the set of solutions $F_1(x),...,F_M(x)$ at any given $x\in\mathbb{R}^n$, analyzing the set of solutions which vary smoothly in $x$ (in particular, lie in $C^m$) is much more difficult, with most progress coming only recently \cite{Luli-JMSJ, FShv18, LULI-2, LULI-3, Feff-Kollar, Luli-IMRN}. 

We begin with a review of the literature on this subject before turning to the main object: $C^m$ solutions for systems of linear \emph{inequalities} (\ref{intro1}).

Regarding \eqref{eq:system of linear equalities},  the simplest question to be asked is the following:

\begin{problem}\label{prob 1 for equalities}{(Brenner-Epstein-Hochster-Koll\'ar Problem)}
Given $A_{ij}, f_i$ as in \eqref{eq:system of linear equalities}, determine if there exists a $C^m$ solution $F=(F_1,...,F_M)$.
\end{problem}

Problem \ref{prob 1 for equalities} was solved by Fefferman and the first author in \cite{FL14} (see also \cite{Feff-Kollar, Luli-JMSJ}), which motivated a number of related works \cite{Luli-JMSJ, bcm, Luli-IMRN, JIANG2022108566}.

Next, one may try to analyze the set of $f=(f_1,...,f_N)\in C^\infty$ for which there exists a $C^m$ solution $F$. For various reasons, it is helpful to consider particular cases of $A_{ij}$, namely \emph{semilagebraic functions}: a function $F: \mathbb{R}^n \rightarrow \mathbb{R}$ is \underline{semialgebraic} if its graph can be represented as the solution set to finitely many polynomial equations and/or inequalities. For instance, rational functions and the indicator function on the circle are semialgebraic while exponential functions are not. (See below for more discussion on the choice of semialgebraic functions for this problem.)

\begin{problem}\label{prob 2 for equalties}
Given semialgebraic $A_{ij}$ as in \eqref{eq:system of linear equalities}, characterize the set of $f\in C^\infty(\R^n,\R^N)$ for which there exists a $C^m$ solution $F$.
\end{problem}

%Let $\mathcal{R}$ denote the ring of polynomials in $n$ real variables and of degree at most $m$. 

To motivate the solution to Problem \ref{prob 2 for equalties}, let us review an example of Epstein and Hochster \cite{EH18}. Consider the single linear equation
\begin{equation}\label{eq:Epstein Hochster example}
    x^2F_1+y^2F_2+xyz^2F_3=f(x,y,z).
\end{equation}

There exist continuous $F_1,F_2,F_3$ satisfying \eqref{eq:Epstein Hochster example} if and only if
\begin{equation}
\left[ 
\begin{array}{l}
f\left( x,y,z\right) =\frac{\partial f}{\partial x}\left( x,y,z\right) =%
\frac{\partial f}{\partial y}\left( x,y,z\right) =0 \\ 
\text{and} \\ 
\frac{\partial ^{2}f}{\partial x\partial y}\left( x,y,z\right) =\frac{%
\partial ^{3}f}{\partial x\partial y\partial z}\left( x,y,z\right) =0%
\end{array}%
\right. 
\begin{array}{l}
\text{for }x=y=0,z\in \mathbb{R} \\ 
\\ 
\text{at }x=y=z=0\text{.}%
\end{array}
\label{intro3.3}
\end{equation}

Note that while no differentiability requirements on the $F_j$ are made, derivatives still show up in conditions on the $f_i$ in \eqref{intro3.3}. This example illustrates the general form of the solution to Problem \ref{prob 2 for equalties} as proven by Fefferman and the first author in \cite{LULI-3}:

\begin{theorem}
\label{theorem1} Fix $m\geq 0$, and let $\left( A_{ij}\left( x\right)
\right) _{1\leq i\leq N,1\leq j\leq M}$ be a matrix of semialgebraic
functions on $\mathbb{R}^{n}$. Then there exist linear partial differential
operators $L_{1},L_{2},\cdots ,L_{\nu_{\max}}$, for which the following hold.

\begin{itemize}
\item Each $L_{\nu }$ acts on vectors $f=\left( f_{1},\cdots ,f_{N}\right)
\in C^{\infty }\left( \mathbb{R}^{n},\mathbb{R}^{N}\right) $, and has the
form 
\begin{equation*}
L_{\nu }f\left( x\right) =\sum_{i=1}^{N}\sum_{\left\vert \alpha \right\vert
\leq \bar{m}}a_{\nu i\alpha }\left( x\right) \partial ^{\alpha }f_{i}\left(
x\right) ,
\end{equation*}%
where the coefficients $a_{\nu i\alpha }$ are semialgebraic. (Perhaps $\bar{m%
}>m$.)

\item Let $f=\left( f_{1},\cdots ,f_{N}\right) \in C^{\infty }\left( \mathbb{%
R}^{n},\mathbb{R}^{N}\right) $. Then the system \eqref{eq:system of linear equalities} admits a $C^{m}\left( \mathbb{R}^{n},%
\mathbb{R}^{M}\right)$
solution $F=\left( F_{1},\cdots ,F_{M}\right)$ if and only if $L_{\nu }f=0$ on $\mathbb{R}^{n}$ for
each $\nu =1,\cdots ,\nu_{\max}$.
\end{itemize}
\end{theorem}

We now turn to the case of inequalities:
\begin{equation}
\sum_{j=1}^{M}A_{ij}\left( x\right) F_{j}\left( x\right) \leq f_{i}\left(
x\right) \text{ }\left( i=1,\cdots ,N\right) \text{ on } \mathbb{R}^n.  \label{intro1}
\end{equation}%

\begin{problem}\label{prob 1 for INequalities}
    Given $A_{ij}, f_i$ as in \eqref{intro1}, determine if there exists a $C^m$ solution $F=(F_1,...,F_M)$.
\end{problem}

Problem \ref{prob 1 for INequalities} is the analogue of Problem \ref{prob 1 for equalities} for inequalities and was solved recently by Jiang and the authors in \cite{JIANG2022108566}.

The focus of this paper is the following analogue of Problem \ref{prob 2 for equalties} for inequalities.

\begin{problem}\label{prob 2 for INequalities}
Fix $m,M,n,N\in\N$ and let $A_{ij}:\R^n\to\R$ be semialgebraic. Characterize the set of $f=(f_1,...,f_N)\in C^\infty(\R^n,\R^N)$ for which there exists a $C^m$ solution $F=(F_1,...,F_M)$ to \eqref{intro1}.
\end{problem}

It is well-known that, for fixed $x$, any linear, convex constraints may be put into the form \eqref{intro1}. For our phrasing of Problem \ref{prob 2 for INequalities}, such equivalence does not hold. For instance, in a linear programming problem, the equality constraint $a\cdot x= b$ is equivalent to requiring both inequalities $a\cdot x\le b$ and $(-1)\cdot x\le -b$. However, in the context of Problem \ref{prob 2 for INequalities}, replacing one constraint with two constraints leads to the presence of an additional $f_i$ and a different problem. A much more general version of Problem \ref{prob 2 for INequalities} could be stated, but this would be unnecessary for the purposes of providing a counterexample.

The recent solution \cite{JIANG2022108566} to Problem \ref{prob 1 for equalities} for \eqref{intro1} by Jiang and the current authors provides a key step to analyze Problem \ref{prob 2 for INequalities}. Much like the solution to Problem \ref{prob 1 for equalities} in \cite{FL14}, it solved Problem \ref{prob 2 for INequalities} in terms of the ``Glaeser refinement
technique", which is a higher-dimensional generalization of the divided difference \cite{G58,BMP03,Fefferman2006Solution}.
%(see the discussion below for more details) %We only define the C^0 version so such discussion may not exist
This work \cite{JIANG2022108566} provides a solution to Problem \ref{prob 2 for INequalities} {\em in principle}, but in practice it is difficult to verify the conditions.

To motivate our expected result on a system of inequalities, let us consider an example. For simplicity, we temporarily ignore the previous discussion and consider systems more general than those described by \eqref{intro1}. Suppose $f \in C^\infty (\mathbb{R})$ and consider the following inequalities
for $x \in \mathbb{R}$, 
\begin{equation}
\begin{cases}
x^2\mathbb{I}_{x \geq 0} F \leq f \leq x \mathbb{I}_{x \geq 0}F \\
x\mathbb{I}_{x \leq 0} F \leq f \leq x^2 \mathbb{I}_{x \leq 0}F \label{intro3a}
\end{cases}\end{equation}
for unknown continuous $%
F$ on $\mathbb{R}$. One
checks that a continuous solution $F$ exists if and only if $f$ satisfies 
\begin{equation}
\left[ 
\begin{array}{l}
f\left( 0 \right) = 0, \\
%f(x) = 0, \forall x \geq 1 \\
f'(0) \geq 0.
\end{array}%
\right. 
\label{intro3}
\end{equation}

Note that the derivative of $f$ enters into (\ref{intro3}), even though
we are merely looking for continuous solutions $F$.

This simple example helps us formulate a result similar to Theorem \ref{theorem1} for a system of inequalities. At its simplest, it says that the existence of a $C^m$ solution may be determined by a finite set of linear partial differential inequalities in the $f_i$.

\begin{conjecture}\label{conj:main conj}
Fix $m\ge0$ and let $(A_{ij}(x))_{1\le i\le N,1\le j\le M}$ be a matrix of semialgebraic functions on $\R^n$. Then, there exist linear partial differential operators $$L_{1,1},...,L_{1,\nu_1},...,L_{\mu_{\max},1},...,L_{\mu_{\max},\nu_{\mu_{\max}}},L'_{1,1},...,L'_{1,\nu_1'},...,L'_{\mu_{\max},1},...,L'_{\mu_{\max},\nu'_{\mu_{\max}}}$$ for which the following hold:
\begin{enumerate}
    \item Each $L_{\mu,\nu}$  acts on vectors $f=(f_1,...,f_N)\in C^\infty(\R^n,\R^N)$ and has the form
    \begin{equation*}
        L_{\mu,\nu} f(x)=\sum_{i=1}^N\sum_{|\alpha|\le \mb} a_{\mu\nu i\alpha}(x)\da f_i(x),
    \end{equation*}
    or
    \begin{equation*}
        L'_{\mu,\nu} f(x)=\sum_{i=1}^N\sum_{|\alpha|\le \mb} a'_{\mu\nu i\alpha}(x)\da f_i(x),
    \end{equation*}
    where the coefficients $a_{\mu\nu i\alpha},a'_{\mu\nu i\alpha}$ are semialgebraic and $\mb\ge m$.
    \item Let $f=(f_1,...,f_N)\in C^\infty(\R^n,\R^N)$. Then the system \eqref{intro1} admits a solution $F=(F_1,...,F_M)\in C^m(\R^n,\R^M)$ if and only if there exists $1\le \mu\le \mu_{\max}$ such that $L_{\mu,\nu} f\ge0$ on $\R^n$ for each $1\le \nu\le \nu_{\mu}$ and $L'_{\mu,\nu}f>0$ on $\R^n$ for each $1\le \nu\le \nu'_{\mu}$.
\end{enumerate}
\end{conjecture}

It may appear natural to simply replace the condition $L_\nu f=0$ in Theorem \ref{theorem1} with $L_\nu f\ge0$, or perhaps $L_\nu f>0$, corresponding to the case $\mu_{\max}=1$ above. However, an attempt to replicate the proof of Theorem \ref{theorem1} with inequalities in place of equations leads naturally to a more general condition. Furthermore, one may interpret the conditions in Conjecture \ref{conj:main conj} as a more general formalization of the idea of ``determined by a finite set of linear partial differential inequalities.''

%Conjecture \ref{conj:main conj} is best viewed in comparison to Theorem \ref{theorem1}.

% For \eqref{intro3a}, with $m=0,$ the operators are 
% \begin{equation*}
% \mathbb{I}_{x=0}(\frac{d}{dx})^0,
% %\mathbb{I}_{x\geq 1}(\frac{d}{dx})^0,
% \mathbb{I}_{x=0}(\frac{d}{dx})
% \end{equation*}%
% where $\mathbb{I}$ denotes the indicator function. (Compare with (\ref%
% {intro3}).)

The main result of this paper is that Conjecture \ref{conj:main conj} is false for $n \geq 2$. A counterexample is given for the case of $C^0(\R^2,\R^2)$. For $n=1$, the conjecture remains open.

The starting point for the construction of our counterexample is that a semialgebraic function may have an infinite number of directional limits at a single point. As a result, computing the Glaeser refinement at that point amounts to taking the infinite intersection of polytopes, which may not itself be a polytope. (This problem is avoided in the solution to Problem \ref{prob 2 for equalties} found in \cite{LULI-3} since the infinite intersection of affine spaces is itself an affine space.) This motivates the design of the counterexample, which is stated fully in Section \ref{sec:counterex}.

It is natural to ask why, if semialgebraic functions can lead to such problems, one does not simply use polynomials in place of semialgebraic functions in \eqref{intro1}, in alignment with the versions stated in \cite{Brenner,EH18}. The reason is that the difference quotients used in Glaeser refinements are semialgebraic functions, and in following the analysis of say, \cite{LULI-3}, any attempt to begin with polynomial coefficients $A_{ij}$ leads to the use of semialgebraic functions anyway. While our counterexample requires the greater generality of semialgebraic functions (versus polynomials) it shows that in order to prove a version of Conjecture \ref{conj:main conj} for the polynomial case new techniques would have to be developed. Furthermore, it would be reasonable for this case to simply require an analogous, yet more complicated counterexample.

We begin with a review of our main computational tool, Glaeser refinement, and its importance in Section \ref{sec:background}. In Section \ref{sec:nonlinear}, we compute the Glaeser refinement for our example manually and determine explicit criteria for the existence of $C^0$ solutions. We use this result to demonstrate the nonexistence of linear criteria in Section \ref{sec:disproof}, officially disproving Conjecture \ref{conj:main conj}.

\subsection*{Acknowledgement} The first author is supported by the UC Davis Chancellor’s Fellowship, the Collaboration Grants for Mathematicians by the Simons Foundation, the National Science Foundation (NSF), grant number DMS-2247429.

%Lastly, we demonstrate how this example may be extended to the case where $n>2$ or $m>0$ in Section \ref{sec:generalize}.

\section{Glaeser Refinement}\label{sec:background}

%\subsection{Glaeser Refinement}

As our counterexample is in the case of continuous functions, we provide the following definition of Glaeser refinement for this special case. (See \cite{Fefferman2006Solution} and \cite{JIANG2022108566} for more general definitions of Glaeser refinement.)

\begin{definition}
If $(K(x))_{x\in E}$ is a collection of subsets of $\R^d$, we define the $C^0$-Glaeser refinement of $(K(x))_{x\in E}$, denoted $(\tilde{K}(x))_{x\in E}$ by
\begin{equation*}
    \tilde{K}(x)=\{z\in K(x): \forall \epsilon>0, \exists \delta>0 \text{ s.t. }y\in B_{\delta}(x)\Rightarrow \exists z'\in K(y), |z-z'|<\epsilon\}.
\end{equation*}
\end{definition}

\begin{theorem}\label{thm:JLO}
There exists $l^*=l^*(n,d)$ such that the following holds.

Let $E\subset\R^n$ be compact $(K(x))_{x\in E}$ be a collection of closed, convex sets in $\R^d$. Let $(K'(x))_{x\in E}$ be the $l^*$-th iterated Glaeser refinement of $(K(x))_{x\in E}$. Then, $(K(x))_{x\in E}$ has a section if and only if $K'(x)$ is nonempty for all $x\in E$.
\end{theorem}

Theorem \ref{thm:JLO} follows somewhat easily from the Michael selection theorem (see \cite{Michael}); however, to spare the reader this work, we cite it as a mere special case of Theorem 1.5 in \cite{JIANG2022108566}. 

% \subsection{Semialgebraic Functions}\label{subsec:semialgebraic}

%May not need since paragraph in introduction may suffice

\section{The counterexample}\label{sec:counterex}

% We consider the case $C^0(\R^2;\R^2)$.

Let $E=[0,1]^2$ and write elements of $E$ as $x=(x_1,x_2)$. For brevity, we will write $0$ in place of $(0,0)$ when usage is clear by context.

Consider the system of equations
\begin{align}
    \frac{x_1^4}{(x_1^2+x_2^2)^2}F_1(x)+\frac{x_2^4}{(x_1^2+x_2^2)^2}F_2(x)&\le f_1(x)\label{eq:early1}\\
    \frac{x_2^4}{(x_1^2+x_2^2)^2}F_1(x)-\frac{x_1^4}{(x_1^2+x_2^2)^2}F_2(x)&\le f_2(x)\label{eq:early2}\\
    -\frac{x_1^4}{(x_1^2+x_2^2)^2}F_1(x)-\frac{x_2^4}{(x_1^2+x_2^2)^2}F_2(x)&\le f_3(x)\label{eq:early3}\\
    -\frac{x_2^4}{(x_1^2+x_2^2)^2}F_1(x)+\frac{x_1^4}{(x_1^2+x_2^2)^2}F_2(x)&\le f_4(x)\label{eq:early4}
\end{align}
for $(x_1,x_2)\in E\setminus\{0\}$ and
\begin{equation}\label{eq: OG system at origin}
    0\le f_1(x),f_2(x),f_4(x), \hspace{.25 in} -10^{6}F_1(x)\le f_3(x)
\end{equation}
for $x=0$.

We may summarize this system in the form
\begin{equation}\label{eq:standard form}
    A(x)F(x)\le f(x),
\end{equation}
where $f=(f_1,f_2,f_3,f_4)$ and $F=(F_1,F_2)$.

% The expression $-10^{-6}F_1(x)\le f_3(x)$ is meant to let us choose $f_3((0,0))$ to be negative in our final analysis (hence no $0\le f_3(0)$ condition). It is also designed to have no impact on the computation of $H_1(0)$ in this case, as this is nearly redundant with \eqref{eq:early3} (see final steps for specifics).

% If one replaces the condition $-10^{6}F_1(x)\le f_3(x)$ at $x=(0,0)$ with $f_3((0,0))\le 0$, then the following analysis will similarly determine the lack of linear criteria in $f_i$ for the existence of a $C^0$ solution. However, this is not allowed in a system of the form \eqref{eq:standard form}. That said, \eqref{eq: OG system at origin} will lead naturally to the 

% An alternative approach would be to allow for both $\le$'s and $\ge$'s in the system, allowing the condition $f_3((0,0))\le 0$ (multiply \eqref{eq:early3} by -1, then sub $f_3$ for $-f_3$). However, despite the added complications, this allows for a counterexample to not just the more general conjecture, but the specific case regarding the system described in \eqref{eq:standard form}.

For $x\neq0$, define
\begin{equation*}
    B(x)=\begin{bmatrix}
\frac{x_1^4}{(x_1^2+x_2^2)^2}&\frac{x_2^4}{(x_1^2+x_2^2)^2}\\
    \frac{x_2^4}{(x_1^2+x_2^2)^2} & -\frac{x_1^4}{(x_1^2+x_2^2)^2}
    \end{bmatrix},
\end{equation*}
so that \eqref{eq:early1},\eqref{eq:early2},\eqref{eq:early3}, and \eqref{eq:early4} together may be rewritten in the form
\begin{equation}\label{eq:off-zero B form}
   \begin{bmatrix}
        -f_3(x)\\
        -f_4(x)
    \end{bmatrix}\le B(x)F(x)\le\begin{bmatrix}
        f_1(x)\\
        f_2(x)
    \end{bmatrix}.
\end{equation}

For later use, we note the trivial fact that for all $x\neq 0$,
\begin{equation}\label{eq:norm B inv}
    \|B(x_1,x_2)^{-1}\|\le 4.
\end{equation}

% EDIT: Smooth out all this and introduce $B(\theta)$ notation, along with $(-f_3,-f_4)\le B(x)\le f_1,f_2$. Then, adjust all the future references to our system of equations. Also, be careful that $A$ doesn't have the same symmetry at the origin.

Define
\begin{equation*}
    H_0(x)=\{y\in\R^2: A(x)y\le f(x)\}
\end{equation*}
and $H_{k+1}(x)=\tilde{H}_{k}(x)$ for $k\ge0$, where $(\tilde{H}(x))_{x\in E}$ is the $C^0$-Glaeser refinement of the bundle $(H(x))_{x\in E}$.\\

\section{Nonlinear Criteria for Characterization}\label{sec:nonlinear}

% \textbf{Step 0:} Characterize when $H_0(x)$ is nonempty for all $x\in E$.

\begin{lemma}\label{lemma:H_0}
Let $x_0\in E$. If $x_0=0$, then $H_0(x_0)$ is nonempty.

Let $x_0\in E\setminus\{0\}$. Then $H_0(x_0)$ is nonempty if and only if
\begin{equation}\label{eq:basic off origin}
    -f_3(x_0)\le f_1(x_0), \hspace{.5 in} -f_4(x_0)\le f_2(x_0).
\end{equation}
\end{lemma}

\begin{proof}
By definition and \eqref{eq: OG system at origin}, $H_0(0)=\{(y_1,y_2):-10^6y_1\le f_3(0)\}$. This is nonempty, independent of the choice of $f_1,...,f_4$.

Let $x_0\in E\setminus\{0\}$. Suppose \eqref{eq:basic off origin} holds and choose $z=(z_1,z_2)$ such that
\begin{equation*}
    -f_3(x_0)\le z_1\le f_1(x_0), \hspace{.5 in} -f_4(x_0)\le z_2\le f_2(x_0).
\end{equation*}

Then, by \eqref{eq:off-zero B form}, $B^{-1}(x_0)z\in H_0(x_0)$, so $H_0(x_0)$ is nonempty. If \eqref{eq:basic off origin} fails, then clearly there is no solution to \eqref{eq:off-zero B form} with $x=x_0$ and $H_0(x_0)$ is empty.
\end{proof}

%\textbf{Step 1:} Compute $H_1(x)$ for $x\in E\setminus\{(0,0)\}$.\\

% We only need consider this in the case that $H_0(p)$ is nonempty for all $p\in E\setminus\{(0,0)\}$. The case where this doesn't hold (a particular failure of our system to have a $C^0$ solution) is already classified. So assume this case.

\begin{lemma}\label{lemma:all nonzero nonempty}
    Suppose $H_0(x)$ is nonempty for all $x\in E\setminus\{0\}$. Then, $H_k(x)=H_0(x)$ for all $x\in E\setminus\{0\}$ and $k\ge0$ and $H_k(0)=H_1(0)$ for all $k\ge1$.
\end{lemma}

\begin{proof}

Fix $x\in E\setminus\{0\}$ and let $y=(y_1,y_2)\in H_0(x)$. Fix $\epsilon>0$ Then,
\begin{equation*}
    \begin{bmatrix}
        -f_3(x)\\
        -f_4(x)
    \end{bmatrix}\le B(x) \begin{bmatrix}
        y_1\\
        y_2
    \end{bmatrix}:=\begin{bmatrix}
        z_1\\z_2
    \end{bmatrix}
    \le\begin{bmatrix}
        f_1(x)\\
        f_2(x)
    \end{bmatrix}.
\end{equation*}

Choose $\delta>0$ such that $x'\in E\cap B_\delta(x)$ implies $|f_i(x')-f_i(x)|<\epsilon/10$ for all $i$ and
\begin{equation}\label{eq:B inv cts}
    \|B(x')^{-1}-B(x)^{-1}\|<\epsilon/10\times\min\{1,1/(\|z\|+1)\},
\end{equation}
where $z=(z_1,z_2)$.

Thus, for such $x'$,
\begin{equation}\label{eq:compare with z_1}
    -f_3(x')-\epsilon/10\le z_1\le f_1(x')+\epsilon/10
\end{equation}
and
\begin{equation}\label{eq:compare with z_2}
    -f_4(x')-\epsilon/10\le z_2\le f_2(x')+\epsilon/10.
\end{equation}

Since $H_0(x')$ is nonempty (by assumption), $-f_3(x')\le f_1(x')$. Furthermore, by \eqref{eq:compare with z_1}, there exists
\begin{equation*}
    -f_3(x')\le z_1'\le f_1(x')
\end{equation*}
satisfying
\begin{equation}\label{eq:z1's close 1}
    |z_1'-z_1|\le\epsilon/10.
\end{equation}

Similarly, by \eqref{eq:compare with z_2} there exists $-f_4(x')\le z_2'\le f_2(x')$ satisfying
\begin{equation}\label{eq:z2's close 1}
    |z_2'-z_2|\le\epsilon/10.
\end{equation}

Let $z'=(z_1',z_2')$ and $y'=B(x')^{-1}(z')\in H_0(x')$. Thus, by \eqref{eq:norm B inv}, \eqref{eq:B inv cts}, \eqref{eq:z1's close 1}, and \eqref{eq:z2's close 1},
\begin{align*}
    |y'-y|&=|B(x')^{-1}(z')-B(x)^{-1}(z)|\\
    &\le |B(x')^{-1}(z')-B(x')^{-1}(z)|+|B(x')^{-1}(z)-B(x)^{-1}(z)|\\
    &=|B(x')^{-1}(z'-z)|+|((B(x')^{-1}-B(x)^{-1})(z)|\\
    &\le \|B(x')^{-1}\|(\epsilon/5)+\epsilon/10\min\{1,1/(\|z\|+1)\}\|z\|\\
    &\le 4\epsilon/5+\epsilon/10<\epsilon.
\end{align*}

We conclude that $y\in H_1(x)$ since $x'\in E\cap B_{\delta}(x)$ was arbitrary and $y'\in H_0(x')$ was as desired. Thus, $H_1(x)=H_0(x)$ for all $x\in E\setminus\{0\}$ and $H_1(x)$ is nonempty for all $x\in E\setminus\{0\}$. One may prove $H_{k+1}(x)=H_k(x)$ for $x\in E\setminus\{0\}$ similarly, from which one may conclude $H_k(x)=H_0(x)$ for $x\in E\setminus\{0\}$ and $k\ge 0$.

An element $y\in H_k(0)$ lies in $H_{k+1}(0)$ if and only if it satisfies a certain condition depending on $H_k(x)$ for $x$ in an arbitrarily small neighborhood of the origin. By hypothesis, all such $H_k(x)$ are the same, so further applications of Glaeser refinement make no difference.
\end{proof}

\begin{corollary}\label{cor:cor}
Let $l^*$ be as in Theorem \ref{thm:JLO}. Then, $H_{l^*}(x)$ is nonempty for all $x\in E$ if and only if $H_0(x)$ is nonempty for all $x\in E\setminus\{0\}$ and $H_1(0)$ is nonempty.
\end{corollary}

By Lemma \ref{lemma:H_0}, we may use \eqref{eq:basic off origin} to categorize when $H_0(x)$ is nonempty. We now move to the case of $H_1(0)$.

% \textbf{Step 2:} Describe $H_0((0,0))$.\\

% If $f_3((0,0))\le 0\le f_1((0,0)),f_2((0,0)),f_4((0,0))$, then $A(x)y=\le f(x)$ at $x=(0,0)$ for all $y\in\R^2$. Thus, $H_0((0,0))=\R^2$. Else, the system of equations is satisfied for no such $y$ and $H_0((0,0))$ is empty.\\

% If $H_0(p)$ is empty for any $p\in E\setminus\{(0,0)\}$, then clearly $H_1(p)$ is empty.

% Else, Step 1 shows that $H_1(p)=H_0(p)$ for $p\in E\setminus\{(0,0)\}$. Thus $H_1(p)$ is nonempty for all $p\in E\setminus\{(0,0)\}$ if and only if $H_0(p)$ is nonempty for all $p\in E\setminus\{(0,0)\}$ if and only if \eqref{eq:basic off origin} holds for all $p\in E\setminus\{(0,0)\}$.

% We pause to note that this is a characterization in terms of linear, partial differential inequalities with semialgebraic coefficients.

% Also, at this point we have only used the fact that $B$ is invertible and continuous. We only need modify the above (a little bit) if we replace $-B$ with another matrix as the other half of $A$.\\

%\textbf{Step 3:} Compute $H_1(0)$.\\

% First, observe that $H_l(0)=H_1(0)$ for all $l\ge1$ since $H_k(x)=H_0(x)$ for all $k\ge0$ and $x$ in a neighborhood of $0$. So this is all we need to compute to determine if the Glaeser refinement terminates in a nontrivial bundle. (We may need just one step for the $C^0$ version anyway, but this is easier than having to prove it.)

Let $y=(y_1,y_2)\in H_0(0)$. We would like to determine if $y\in H_1(0)$. 

Write $x\in E\setminus\{0\}$ in polar coordinates as $x=r\theta$. Noting that $B(r\theta)$ depends solely on $\theta$, we introduce the notation
\begin{equation*}
    B(\theta)=\begin{bmatrix}
\cos^4\theta &\sin^4\theta\\
    \sin^4\theta & -\cos^4\theta
    \end{bmatrix},
\end{equation*}
so $B(\theta)=B(x)$ for $x=r\theta$.

Thus, $y\in H_0(r\theta)$ if and only if
\begin{equation}\label{eq:B r theta conditions}
    \begin{bmatrix}
        -f_3(r\theta)\\
        -f_4(r\theta)
    \end{bmatrix}\le B(\theta) \begin{bmatrix}
        y_1\\
        y_2
    \end{bmatrix}
    \le\begin{bmatrix}
        f_1(r\theta)\\
        f_2(r\theta)
    \end{bmatrix}.
\end{equation}

\begin{lemma}\label{lemma:characterize H_1(0)}
    Suppose $H_0(x)$ is nonempty for all $x\in E$. Then,
    \begin{equation*}
        H_1(0)=\left\{y\in \R^2: \begin{bmatrix}
        -f_3(0)\\
        -f_4(0)
    \end{bmatrix}\le B(\theta) \begin{bmatrix}
        y_1\\
        y_2
    \end{bmatrix}
    \le\begin{bmatrix}
        f_1(0)\\
        f_2(0)
    \end{bmatrix} \text{ for all }\theta\in [0,\pi/2]\right\}\cap H_0(0).
    \end{equation*}
\end{lemma}

\begin{proof}
    First suppose $y\in H_0(0)$ and
    \begin{equation}\label{eq:all theta at zero}
        \begin{bmatrix}
        -f_3(0)\\
        -f_4(0)
    \end{bmatrix}\le B(\theta) \begin{bmatrix}
        y_1\\
        y_2
    \end{bmatrix}
    \le\begin{bmatrix}
        f_1(0)\\
        f_2(0)
    \end{bmatrix} \text{ for all }\theta\in [0,\pi/2].
    \end{equation}

Fix $\epsilon>0$. Choose $\delta>0$ such that $x\in E\cap B_\delta(0)$ implies
\begin{equation}\label{eq:all f_i close A}
    |f_i(x)-f_i(0)|<\epsilon/10 \text{ for all } i.
\end{equation}

Let $\theta_0\in[0,\pi/2]$ and $0<r<\delta$ (that is, $x\in E\cap B_\delta(0)$). By \eqref{eq:all theta at zero},
\begin{equation*}
    \begin{bmatrix}
        -f_3(0)\\
        -f_4(0)
    \end{bmatrix}\le B(\theta_0) \begin{bmatrix}
        y_1\\ y_2
    \end{bmatrix}:=\begin{bmatrix}
        z_1\\ z_2
    \end{bmatrix}
    \le\begin{bmatrix}
        f_1(0)\\
        f_2(0)
    \end{bmatrix}
\end{equation*}

Since $0<r<\delta$, we use \eqref{eq:all f_i close A} to obtain
\begin{equation*}
    -f_3(r\theta_0)-\epsilon/10\le z_1\le f_1(r\theta_0)+\epsilon/10
\end{equation*}
and
\begin{equation*}
    -f_4(r\theta_0)-\epsilon/10\le z_2\le f_2(r\theta_0)+\epsilon/10.
\end{equation*}

Since $H_0(r\theta_0)$ is nonempty (by assumption), $-f_3(r\theta_0)\le f_1(r\theta_0)$. Furthermore, by \eqref{eq:all f_i close A}, there exists
\begin{equation*}
    -f_3(r\theta_0)\le z_1'\le f_1(r\theta_0)
\end{equation*}
satisfying
\begin{equation}\label{eq:z1's close}
    |z_1'-z_1|\le\epsilon/10.
\end{equation}

Similarly, there exists $-f_4(r\theta_0)\le z_2'\le f_2(r\theta_0)$ satisfying
\begin{equation}\label{eq:z2's close}
    |z_2'-z_2|\le\epsilon/10.
\end{equation}

Let $y'=B(r\theta_0)^{-1}(z_1',z_2')\in H_0(r\theta_0)$. Thus, by \eqref{eq:norm B inv}, \eqref{eq:z1's close}, and \eqref{eq:z2's close},
\begin{align*}
    |y'-y|&=|B(\theta_0)^{-1}(z_1',z_2')-B(\theta_0)^{-1}(z_1,z_2)|\\
    &\le \|B(\theta_0)^{-1}\|\cdot \|(z_1',z_2')-(z_1,z_2)||\\
    &\le 4(\epsilon/5)<\epsilon.
\end{align*}

Therefore, by defintion, $y\in H_1(0)$.

% \begin{align}
%     \cos^4\theta_0 y_1+\sin^4\theta_0 y_2&\le f_1(0)\\
%     \sin^4\theta_0 y_1-\cos^4\theta_0 y_2&\le f_2(0)\\
%     -\cos^4\theta_0 y_1-\sin^4\theta_0 y_2&\le f_3(0)\\
%     -\sin^4\theta_0 y_1+\cos^4\theta_0 y_2&\le f_4(0).
% \end{align}

% Since $H_0(r\theta_0)$ is nonempty for all $0<r<1$, the system
% \begin{equation}
%     \begin{bmatrix}
%         -f_3(r\theta_0)\\
%         -f_4(r\theta_0)
%     \end{bmatrix}\le B(\theta_0) \begin{bmatrix}
%         y_1\\
%         y_2
%     \end{bmatrix}
%     \le\begin{bmatrix}
%         f_1(r\theta_0)\\
%         f_2(r\theta_0)
%     \end{bmatrix}
% \end{equation}
% has a solution for all $0<r<1$.

Now suppose that $y\in H_1(0)$. Then, $y\in H_0(0)$ so we need only check \eqref{eq:all theta at zero}. 

By definition, for all $\epsilon>0$, there exists $\delta_0>0$ such that \eqref{eq:B r theta conditions} has a solution $z=(z_1,z_2)$ satisfying $|z-y|<\epsilon$ whenever $0<r<\delta_0$ and $\theta\in[0,\pi/2]$. Here, $z\in H_0(r\theta)$.

Let $\theta_0\in[0,\pi/2]$ and $\epsilon>0$. Choose $\delta_0$ as above and $0<\delta<\delta_0$ such that $x\in E\cap B_\delta(0)$ implies
\begin{equation*}
    |f_i(x)-f_i(0)|<\epsilon \text{ for all } i.
\end{equation*}

Let $0<r<\delta$. As a particular case of \eqref{eq:B r theta conditions} and the fact $|y-z|<\epsilon$,
 \begin{equation*}
        \cos^4\theta y_1+\sin^4\theta y_2\le \cos^4\theta z_1+\sin^4\theta z_2 +2\epsilon \le f_1(0)+3\epsilon.
    \end{equation*}
Since $\epsilon>0$ was arbitrary,
    \begin{equation*}
        \cos^4\theta y_1+\sin^4\theta y_2\le f_1(0).
    \end{equation*}

By similar arguments involving $f_2,f_3$, and $f_4$ one at a time, we obtain \eqref{eq:all theta at zero}

\end{proof}

% \begin{align}
%     \cos^4\theta z_1+\sin^4\theta z_2&\le f_1(r\theta)\\
%     \sin^4\theta z_1-\cos^4\theta z_2&\le f_2(r\theta)\\
%     -\cos^4\theta z_1-\sin^4\theta z_2&\le f_3(r\theta)\\
%     -\sin^4\theta z_1+\cos^4\theta z_2&\le f_4(r\theta).
% \end{align}

% We claim this is true precisely when 
% \begin{align}
%     \cos^4\theta y_1+\sin^4\theta y_2&\le f_1(0)\\
%     \sin^4\theta y_1-\cos^4\theta y_2&\le f_2(0)\\
%     -\cos^4\theta y_1-\sin^4\theta y_2&\le f_3(0)\\
%     -\sin^4\theta y_1+\cos^4\theta y_2&\le f_4(0)
% \end{align}
% for all $\theta\in[0,\pi/2]$. Define

% so we have
% \begin{equation}
%     \begin{bmatrix}
%         -f_3(0)\\
%         -f_4(0)
%     \end{bmatrix}\le B(\theta) \begin{bmatrix}
%         y_1\\
%         y_2
%     \end{bmatrix}
%     \le\begin{bmatrix}
%         f_1(0)\\
%         f_2(0)
%     \end{bmatrix}.
% \end{equation}
% for all $\theta\in[0,\pi/2]$.

% First suppose the condition on $y$ holds. Choose $\delta>0$ such that $|f_i(x)-f_i(0)|<\epsilon/10$ whenever $r\theta \in E\cap B_\delta(0)$. 

% \textbf{Step 4:} Characterize when $H_1((0,0))$ is nonempty.\\

Considering the conclusion of Lemma \ref{lemma:characterize H_1(0)} and the fact that $H_1(0)\subset H_0(0)$ trivially, the main question at hand is which $(y_1,y_2)$ satisfy \eqref{eq:all theta at zero}, that is,
\begin{align*}
    \cos^4\theta y_1+\sin^4\theta y_2&\le f_1(0)\\
    \sin^4\theta y_1-\cos^4\theta y_2&\le f_2(0)\\
    -\cos^4\theta y_1-\sin^4\theta y_2&\le f_3(0)\\
    -\sin^4\theta y_1+\cos^4\theta y_2&\le f_4(0)
\end{align*}
for all $\theta\in[0,\pi/2]$.

First consider the set
\begin{equation*}
    R_1:=\{y\in\R^2: \cos^4\theta y_1+\sin^4\theta y_2\le f_1(0) \text{ for all }\theta\in [0,\pi/2]\}.
\end{equation*}

Plugging in $\theta=0$ and $\theta=\pi/2$, we have $y_1\le f_1(0)$ and $y_2\le f_1(0)$ as defining constraints for $R_1$. Since $f_1(0)\ge0$ by \eqref{eq: OG system at origin} and $\sin^4\theta+\cos^4\theta\le 1$ for all $\theta$, these two inequalities imply all the rest and
\begin{equation}\label{eq:R1 final form}
    R_1=\{(y_1,y_2):y_1\le f_1(0),y_2\le f_1(0)\}
\end{equation}

Similarly,
\begin{align}
    R_2:&=\{y\in\R^2: \sin^4\theta y_1-\cos^4\theta y_2\le f_2(0) \text{ for all }\theta\in [0,\pi/2]\}\\
    &=\{(y_1,y_2):y_1\le f_2(0),y_2\ge -f_2(0)\}\label{eq:R2 final form}
\end{align}
and
\begin{align}
    R_4:&=\{y\in\R^2: -\sin^4\theta y_1+\cos^4\theta y_2\le f_4(0) \text{ for all }\theta\in [0,\pi/2]\}\\
    &=\{(y_1,y_2):y_1\ge -f_4(0),y_2\le f_4(0)\}\label{eq:R4 final form}
\end{align}

The region
\begin{equation*}
    R_3:=\{y\in\R^2: -\cos^4\theta y_1-\sin^4\theta y_2\le f_3(0) \text{ for all }\theta\in [0,\pi/2]\}
\end{equation*}
will not be described so easily since the value of $f_3(0)$ is allowed to be negative, that is, $R_3$ is not determined by $-\cos^4\theta y_1-\sin^4\theta y_2\le f_3(0)$ for just two choices of $\theta$.

So suppose here that $f_3(0)<0$.
%If the existence of a $C^0$ solution could be determined via a finite number of partial differential inequalities in general, then that would also be true within the specific case defined by the linear inequality $f_3(0)<0$.

% To compute
% \begin{equation}
%     R_3:=\{y\in\R^2: -\cos^4\theta y_1-\sin^4\theta y_2\le f_3(0) \text{ for all }\theta\in [0,\pi/2]\}
% \end{equation}
% requires a little more care since the right-hand side of the inequality is negative; as a result, $R_3$ is not determined by the given inequality for just two choices of $\theta$.

Substituting $M=-f_3(0)>0$ and $a=\sin^2\theta$, we have
\begin{align*}
    R_3&=\{y\in\R^2: (1-a)^2y_1+a^2y_2\ge M \text{ for all }a\in[0,1]\}\\
    &= \{y\in\R^2: y_2\ge \frac{M-(1-a)^2y_1}{a^2} \text{ for all }a\in(0,1], y_1\ge M\}\\
    &=\bigcup_{y_1>M}\{(y_1,y_2):y_2\ge \frac{M-(1-a)^2y_1}{a^2} \text{ for all }a\in(0,1]\}.
\end{align*}

Given $y_1$, we find the largest value of $\frac{M-(1-a)^2y_1}{a^2}$ ranging over all $a\in(0,1]$; call this value $W(y_1)$. It follows from the above that
\begin{align*}
    R_3&=\bigcup_{y_1>M}\{(y_1,y_2):y_2\ge W(y_1)\}\\
    &=\{y\in R^2:y_1>M, y_2\ge W(y_1)\}.
\end{align*}

To compute $W(y_1)$, we define
\begin{equation*}
    V(y_1,a)=\frac{M-(1-a)^2y_1}{a^2}=\frac{M}{a^2}-\frac{y_1}{a^2}+\frac{2y_1}{a}-y_1
\end{equation*}
and find its maximum in $a$. By elementary calculus,
\begin{equation*}
    \frac{\partial V}{\partial a}=\frac{-2M}{a^3}+\frac{2y_1}{a^3}-\frac{2y_1}{a^2}
\end{equation*}
and
\begin{equation*}
    \frac{\partial^2 V}{(\partial a)^2}=\frac{6M}{a^4}-\frac{6y_1}{a^4}+\frac{4y_1}{a^3}.
\end{equation*}

Solving $\frac{\partial V}{\partial a}=0$ for $a$ gives $a=1-\frac{M}{y_1}$. By simple computation, $\frac{\partial^2 V}{(\partial a)^2}<0$ at $a=1-\frac{M}{y_1}$ so this is indeed a local maximum. As the only critical point, it is the global maximum. (In order for $a=0$ or $a=1$ to compete, there would need to be a local minimum between $a=1-\frac{M}{y_1}$ and $a=0$ or $a=1$, but we have already found all the critical points.)

We find
\begin{align*}
    W(y_1)&=V\left(y_1,1-\frac{M}{y_1}\right)\\
    &=\frac{M-(1-1-\frac{M}{y_1})^2y_1}{(1-\frac{M}{y_1})^2}\\
    &=\frac{My_1}{y_1-M}=M+\frac{M^2}{y_1-M}.
\end{align*}

We conclude that
\begin{equation}\label{eq:R3 final form}
    R_3=\{y\in R^2:y_1>M, y_2\ge M+\frac{M^2}{y_1-M}\}.
\end{equation}

In words, $R_3$ is the region contained in the upper-right quadrant of the plane with boundary given by a the upper-right component of a hyperbola with asymptotes $y_1=M$ and $y_2=M$.

Thus, we have established the following:
\begin{lemma}\label{lemma:H_1(0)}
Suppose $H_0(x)$ is nonempty for all $x\in E\setminus\{0\}$ and $f_3(0)<0$. Then,
\begin{equation}\label{eq:H_1 at origin}
    H_1(0)=\{y:y_1\ge -10^{-6}f_3(0)\}\cap (R_1\cap R_3)\cap (R_2\cap R_4),
\end{equation}
where $R_1,R_2,R_3,R_4$ are explicitly described in \eqref{eq:R1 final form}, \eqref{eq:R2 final form}, \eqref{eq:R3 final form}, and \eqref{eq:R4 final form}, respectively.
\end{lemma}

Putting together Lemmas \ref{lemma:H_0} and \ref{lemma:H_1(0)}, Corollary \ref{cor:cor}, and Theorem \ref{thm:JLO} we have:
\begin{proposition}\label{prop:proven criteria}
Let $f_1,f_2,f_3,f_4\in C^\infty(\R^2,\R)$ such that $f_3(0)<0$. Then, \eqref{eq:standard form} has a $C^0$ solution if and only if
\begin{equation}\label{eq:proven criteria}
    -f_3(x)\le f_1(x) \text{ and } -f_4(x)\le f_2(x) \text{ for all } x\in E\setminus\{0\};
\end{equation}
and $H_1(0)$, as specified in \eqref{eq:H_1 at origin}, is nonempty.

\end{proposition}

We now restrict to the case where all the $f_i$ are constant. Define
\begin{equation*}
    K=\{f\in C^{\infty}(\R^2;\R^4): f_1,f_2,f_3,f_4 \text{ are constant}\}.
\end{equation*}

Furthermore, define
\begin{equation*}
    K_0=\{f\in K: \eqref{eq:standard form} \text{ has a solution}, f_3\le -.1\}.
\end{equation*}

By Proposition \ref{prop:proven criteria},
\begin{multline}\label{eq:Proven Criteria}
    K_0=\{f\in K: f_1,f_2,f_4\ge0, f_3\le-0.1, f_1+f_3\ge0, f_2+f_4\ge0\}\\
    \cap \{f\in K: \{y:y_1\ge -10^{-6}f_3\}\cap (R_1\cap R_3)\cap (R_2\cap R_4)\}\neq\emptyset
\end{multline}
The above can be made sense of through the fact that the $R_i$ depend on $f$ in their definitions.

Viewing $K$ as a four-dimensional Hilbert space, we claim $K_0$ is not a polytope. To see this, restrict further to the affine subspace where $f_3(0)=-1$ and $f_1(0)=3$. Thus,
\begin{equation*}
    R_1=\{(y_1,y_2):y_1\le 3,y_2\le 3\}
\end{equation*}
and
\begin{equation*}
    R_3=\{y\in \R^2:y_1>1, y_2\ge 1+\frac{1}{y_1-1}\}.
\end{equation*}

One may readily see that $\{y:y_1\ge -10^{-6}f_3(0)\}=\{y:y_1\ge 10^{-6}\}$ contains $R_1\cap R_3$ so this restriction is superfluous and we need only consider whether $\cap_i R_i$ is nonempty.

Since
\begin{equation*}
     R_2\cap R_4=\{(y_1,y_2):-f_4(0)\le y_1\le f_2(0),-f_2(0)\le y_2\le f_4(0)\}
\end{equation*}
and $f_2(0),f_4(0)\ge0$, the question becomes whether the upper right corner $(f_2(0),f_4(0))$ meets $R_1\cap R_3$. In the range of $1\le f_2(0),f_4(0)\le 3$, this is a nonlinear problem since $R_1\cap R_3$ has a curved boundary given by $y_2\ge 1+\frac{1}{y_1-1}$. Thus, $K_0$ may not be defined by a finite number of linear inequalities.

\begin{lemma}\label{lemma:not a polytope}
    The set $\tilde{K}$ of $f\in K$ such that \eqref{eq:standard form} has a $C^0$ solution may not be defined by finitely many linear inequalities.
\end{lemma}

\begin{proof}
    Suppose $\tilde{K}$ may be defined by finitely many linear inequalities. Then $\tilde{K}\cap \{f:f_3\le -0.1\}$ may be defined by finitely many linear inequalities. However, $K_0=\tilde{K}\cap \{f:f_3\le -0.1\}$, contradicting our above reasoning.
\end{proof}

\section{Disproof of Conjecture}\label{sec:disproof}

%\textbf{Step 5:} Complete the argument by showing how nonlinear criteria cannot lead to linear criteria.\\

So far, we have found nonlinear criteria on $f$ for the existence of a $C^0$ solution $F$ to the system \eqref{eq:standard form}. However, this does not automatically show that there do not exist linear criteria.

% At this point, we have established the following: Let $f\in C^{\infty}(\R^2;\R^4)$ be constant. There is a $C^0$ solution to $AF\le f$ if and only if 
% \begin{itemize}
%     \item $f_1(0)+f_3(0)\ge0$ and $f_2(0)+f_4(0)\ge0$;
%     \item $f_1(0),f_2(0),f_4(0)\ge0$; and
%     \item $H_1((0,0))=\{y:y_1\ge -10^{-6}f_3(0)\}\cap (R_1\cap R_3)\cap (R_2\cap R_4)$ is nonempty.
% \end{itemize}
% We refer to the above as the ``Proven Criteria."

% While these criteria are in the form of nonlinear partial differential inequalities, we must still demonstrate that the set of functions satisfying these may not be determined by linear partial differential inequalities.

Suppose, for the sake of contradiction, that Conjecture \ref{conj:main conj} holds. That is, there exist linear partial differential operators $$L_{1,1},...,L_{1,\nu_1},...,L_{\mu_{\max},1},...,L_{\mu_{\max},\nu_{\mu_{\max}}},L'_{1,1},...,L'_{1,\nu'_1},...,L'_{\mu_{\max},1},...,L'_{\mu_{\max},\nu'_{\mu_{\max}}}$$ for which the following hold:
\begin{enumerate}
    \item Each $L_{\mu,\nu}$  acts on vectors $f=(f_1,...,f_N)\in C^\infty(\R^n,\R^N)$ and has the form
    \begin{equation*}
        L_{\mu,\nu} f(x)=\sum_{i=1}^N\sum_{|\alpha|\le \mb} a_{\mu\nu i\alpha}(x)\da f_i(x),
    \end{equation*}
    or
    \begin{equation*}
        L'_{\mu,\nu} f(x)=\sum_{i=1}^N\sum_{|\alpha|\le \mb} a'_{\mu\nu i\alpha}(x)\da f_i(x),
    \end{equation*}
    where the coefficients $a_{\mu\nu i\alpha},a'_{\mu\nu i\alpha}$ are semialgebraic and $\mb\ge m$.
    \item Let $f=(f_1,...,f_N)\in C^\infty(\R^n,\R^N)$. Then the system \eqref{intro1} admits a solution $F=(F_1,...,F_M)\in C^m(\R^n,\R^M)$ if and only if there exists $1\le \mu\le \mu_{\max}$ such that $L_{\mu,\nu} f\ge0$ on $\R^n$ for each $1\le \nu\le \nu_{\mu}$ and $L'_{\mu,\nu}f>0$ on $\R^n$ for each $1\le \nu\le \nu'_{\mu}$.
\end{enumerate}
We refer to the above as the ``Supposed Criteria."

For $f=(f_1,...,f_4)\in K$, there is a solution to our system if and only if there exists $1\le \mu\le \mu_{\max}$ such that
\begin{equation}\label{eq:pointwise inequalities}
        \sum_{i=1}^N a_{\mu\nu i}(x) f_i\ge0
\end{equation}
for all $x\in E$ and $1\le \nu\le \nu_{\mu}$ and 
\begin{equation}\label{eq:strict pointwise inequalities}
        \sum_{i=1}^N a'_{\mu\nu i}(x) f_i>0
\end{equation}
for all $x\in E$ and $1\le \nu\le \nu'_{\mu}$, where $a_{\mu\nu i}=a_{\mu\nu i0}$ and $a'_{\mu\nu i}=a'_{\mu\nu i0}$.

For $x\in E$, let
\begin{equation*}
    R_x=\{f\in K: \exists 1\le \mu\le \mu_{\max} \text{ such that } \eqref{eq:pointwise inequalities} \text{ and }\eqref{eq:strict pointwise inequalities}\text{ hold}\}.
\end{equation*}
By definition,
\begin{equation}\label{eq:K_0}
    \tilde{K}=K\cap (\cap_{x\in E}R_x)).
\end{equation}

The immediate concern here is that while each $R_x$ may be defined by a finite number of linear inequalities, the infinite intersection found in \eqref{eq:K_0} may give rise to a set which may not be defined by a finite number of linear inequalities. However, the following lemma demonstrates some redundancy in the inequalities defined in \eqref{eq:pointwise inequalities} and \eqref{eq:strict pointwise inequalities}.

\begin{lemma}\label{lemma:R_x}
For $x\in E\setminus\{0\}$,
\begin{equation*}
    R_x\supset \{f\in K:f_1+f_3\ge0,f_2+f_4\ge0\}.
\end{equation*}
\end{lemma}

In other words, the Supposed Criteria applied away from the origin may be no stricter on the set of constant functions than Proposition \ref{prop:proven criteria}.

\begin{proof}
    Let $x\in E\setminus\{0\}$ and suppose for the sake of contradiction that there exists $f\in K$ satisfying $f_1+f_3\ge0,f_2+f_4\ge0$ yet not lying in $R_x$.

    %Define $f_i(x)\equiv b_i$ for $i=1,2,3,4$.
    Choose $0<r<|x|$ and $\theta\in C_0^\infty(\R^2;\R)$ such that $\theta\ge0$, $\theta\equiv1$ on $B_{r/2}(x)$, and the support of $\theta$ is contained in $B_r(x)$. The zero function ($f_1\equiv...\equiv f_4\equiv0$) is trivially a solution to \eqref{eq:standard form}. By Proposition \ref{prop:proven criteria}, $\theta(f_1,f_2,f_3,f_4)$ is a $C^0$ solution to \eqref{eq:standard form}, as mutliplcation by nonegative scalars preserves \eqref{eq:proven criteria} and the computation of $H_1(0)$ is the same as for the zero function due to the truncated support of $\theta$.
    
    However, $\theta(f_1,f_2,f_3,f_4)$ does not satisfy the Supposed Criteria, at least at the point $x$. This is a contradiction.
\end{proof}

\begin{corollary}\label{cor:almost there}
    \begin{equation*}
        \tilde{K}=K\cap \{f\in K:f_1+f_3\ge0,f_2+f_4\ge0\}\cap R_{0}.
    \end{equation*}
\end{corollary}

\begin{proof}
    By the Proven Criteria,
    \begin{equation*}
        \tilde{K}\subset \{f\in K:f_1+f_3\ge0,f_2+f_4\ge0\},
    \end{equation*}
    so by \eqref{eq:K_0},
    \begin{equation*}
        \tilde{K}=K\cap \{f\in K:f_1+f_3\ge0,f_2+f_4\ge0\}\cap (\cap_{x\in E} R_x)).
    \end{equation*}
    
    Thus, by Lemma \ref{lemma:R_x},
    \begin{equation*}
        \tilde{K}=K\cap \{f\in K:f_1+f_3\ge0,f_2+f_4\ge0\}\cap R_{0}.
    \end{equation*}
\end{proof}

By Corollary \ref{cor:almost there}, $\tilde{K}$ may be defined via finitely many linear inequalities. However, by Lemma \ref{lemma:not a polytope} this is a contradiction. Therefore, we must reject the assumption that the Supposed Criteria exist and conclude that the set of $C^0$ solutions to our system of equations may not be characterized by a finite set of partial differential inequalities. This concludes the proof of our counterexample to Conjecture \ref{conj:main conj}.

The extension to the case $n>2$ is trivial as one may consider $\tilde{f}_i(x_1,...,x_n)=f_i(x_1,x_2)$ in place of $f_i(x_1,x_2)$ in the example and repeat the above analysis. Any $C^0$ solution $F(x_1,x_2)$ from the fully analyzed case extends naturally to a solution $C^0$ solution $\tilde{F}(x_1,...,x_n)=F(x_1,x_2)$. Similarly, any $C^0$ solution to the $n>2$ case of the form $\tilde{F}(x_1,...,x_n)$ automatically restricts to a $C^0$ solution $F(x_1,x_2)=\tilde{F}(x_1,x_2,0,...,0)$.

\bibliographystyle{plain}
\bibliography{bib}

\end{document}